\documentclass[12pt,a4paper]{article}

\usepackage{amsfonts}
\usepackage{amsmath}
\usepackage{amssymb}
\usepackage{amstext}
\usepackage{amsthm}  
\usepackage{mathrsfs} 

\newcommand{\R}{\mathbb R}

\newcommand{\eps}{\varepsilon}

\newtheorem{theorem}{Theorem}[section]
\newtheorem{lemma}[theorem]{Lemma}
\newtheorem{proposition}[theorem]{Proposition}
\newtheorem{definition}[theorem]{Definition}

\theoremstyle{remark}
\newtheorem{remark}[theorem]{Remark}

\DeclareMathOperator{\cat}{cat}

\DeclareMathOperator{\supt}{supt}

\numberwithin{equation}{section}

\begin{document}

\title{Positive solutions for singularly perturbed nonlinear elliptic problem on manifolds
via {Morse} theory}
\author{Marco Ghimenti
\thanks{Dipartimento di Matematica e Applicazioni,
Universit\`a degli Studi di Milano Bicocca,
Via Cozzi, 53, Milano, ITALY. e-mail:
\texttt{marco.ghimenti@unimib.it}}
\and Anna Maria Micheletti
\thanks{Dipartimento di Matematica Applicata,
Universit\`a degli Studi di Pisa, Via F. Buonarroti 1/c, Pisa,
ITALY. e-mail: \texttt{a.micheletti@dma.unipi.it}}}

\date{\em{Dedicated to Franco Nicolosi}}
\maketitle

\begin{abstract}
Given $(M,g_0)$ we consider the problem
$-\eps^2\Delta_{g_0+h}u+u=(u^+)^{p-1}$ with $(\eps,h)\in (0,\bar\eps)\times\mathscr{B}_\rho$. 
Here $\mathscr{B}_\rho$ is a ball centered at $0$ with radius $\rho$ in the Banach space of 
all $C^k$ symmetric covariant $2$-tensors on $M$. Using the Poincar\'e polynomial of $M$, 
we give an estimate on the number of nonconstant
solutions with low energy for $(\eps,h)$ belonging to an residual subset of 
$(0,\bar\eps)\times\mathscr{B}_\rho$, for $\bar\eps$, $\rho$ small enough.


\end{abstract}

\section{Introduction}

Let $(M,g)$ be a smooth compact connected Riemannian manifold of dimension $n\geq2$ 
without boundary, endowed with the metric tensor $g$. We are interested in the following 
problem
\begin{equation}\label{eq1-1}
\left\{
\begin{array}{cl}
-\eps^2\Delta_gu+u=|u|^{p-2}u&\text{in }M\\
u\in H^1_g(M),&u>0.
\end{array}
\right.
\end{equation}
where $2<p<\frac{2n}{n-2}$ with $n\geq 3$, $p>2$ if $n=2$,
and $\eps$ is a positive parameter. Here $H^1_g(M)$ is the completion of $C^\infty(M)$
with respect to the norm $\displaystyle ||u||_g^2=\int_M|\nabla_g u|^2+u^2d\mu_g$.

It is well known that any critical point of the energy functional $J_{\eps,g}:H^1_g\to\R$ 
constrained to the Nehari manifold $N_{\eps,g}$ is a solution of (\ref{eq1-1}). Here 
\begin{equation*}
\begin{split}
J_{\eps,g}(u)&=\frac 1{\eps^n}\int_M\left(\frac{\eps^2}{2}|\nabla_g u|^2+\frac12 u^2
-\frac1p  (u^+)^p\right)d\mu_g\\
N_{\eps,g}&=\{u\in H^1_g(M)\smallsetminus 0 \ :\ J'_{\eps,g}(u)[u]=0\}.
\end{split}
\end{equation*}

A lot of work has been devoted to the problem (\ref{eq1-1}) in various kinds of subsets of $\R^n$.
We limit ourselves to citing the pioneering papers \cite{BaCo88,BaL90,BaLi97,BC91,BC94,Da88}.

In \cite{BP05} the authors shows that the least energy solution of (\ref{eq1-1}), i.e. the 
minimum of $J_{\eps,g}$ on $N_{\eps,g}$, is a positive solution with a spike layer, whose peak 
converges to the maximum point of the scalar curvature $S_g$ of $(M,g)$ as $\eps$ goes to zero.
Both topology and geometry influence the multiplicity of positive solution of problem (\ref{eq1-1}). 
Recently in \cite{DMP09,MP09b,MP09} it has been proved that the existence of positive solutions 
is strongly related to the geometry of $M$, that is stable critical points of the scalar curvature $S_g$ 
generate positive solutions with one ore more peaks as $\eps$ goes to zero.
Previously in \cite{BBM07} (see also \cite{H09,V08}) the authors point out that the topology of 
$M$ has effect on the number of solutions of (\ref{eq1-1}), that is (\ref{eq1-1}) has at least $\cat M$ 
nonconstant solutions for $\eps$ small enough. Here $\cat M$ is the Lusternik Schnirelmann
category of $M$. Moreover in \cite{BBM07} the Poincar\'e polynomial is considered
(see Definition \ref{poinpol}) and the authors assume that 
\begin{equation}
\text{all the solution of the problem (\ref{eq1-1}) are nondegenerate.} 
\end{equation}
Then they prove that problem (\ref{eq1-1}) has at least $2P_1(M)-1$ solutions.

Our main result reads as following.
\begin{theorem}\label{th1-1}
Given $g_0\in \mathscr{M}^k$, the set 
\begin{equation*}
D=\left\{
\begin{array}{c} 
(\eps,h)\in(0,\tilde\eps)\times\mathscr{B}_{\tilde\rho}:
\text{ the problem }-\eps^2\Delta_{g_0+h}u+u=(u^+)^{p-1}\\
\text{has at least }P_1(M)\text{ nonconstant solutions }u 
\text{ with } J_{\eps,g_0+h}(u)<2m_\infty
\end{array}
\right\}
\end{equation*}
is an residual subset in $(0,\tilde\eps)\times\mathscr{B}_{\tilde\rho}$, for $\tilde\eps$ 
and $\tilde\rho$ chosen small enough. 
\end{theorem}

Here $\mathscr{S}^k$ is the space of all $C^k$ symmetric covariant $2$-tensors on $M$ 
and $\mathscr{M}^k$ is the set of all $C^k$ Riemannian metrics on $M$ with $k\geq2$. 
The set $\mathscr{B}_\rho$ is the ball centered at $0$ with radius $\rho$ in the Banach space 
$\mathscr{S}^k$. The number $m_\infty$ is defined by 
\begin{displaymath}
m_\infty=\inf\{J_\infty(v)\ :\ J_\infty'(v)v=0\text{ and }v\neq0\}
\end{displaymath}
where
$\displaystyle J_\infty(v)=\int_{\R^n}\frac12 (|\nabla v|^2+v^2)-\frac 1p|v|^pdx$. 

The paper is organized as follows. In Section \ref{sec2} we fix some notations and we recall 
some results which will be crucial in the proof of main result. In Section \ref{sec3} we prove 
the main result, using some technical results showed in sections \ref{sec4}, \ref{sec5}.

\section{Notation, definition, known results}\label{sec2}
Through this paper we will use the following notations
\begin{itemize}
\item $B(0,R)$ is the ball in $\R^n$ of center $0$ and radius $R$.
\item $B_g(q,R)$ is the geodesic ball in $M$ of center $q$ 
and radius $R$ with the 
distance given by the metric $g$.
\item $\mathscr{ B}_\rho$ is the ball in  the Banach space $\mathscr{ S}^k$ 
 of center $0$ and radius $\rho$.
\item $I(u,r)$ is the ball in $H^1_g$ of center $u$ and radius $r$

\item For $u\in H^1_g(M)$ we use the norms
\begin{eqnarray*}
||u||_{g}^2=\int_M(|\nabla_g u|^2+|u|^2)d\mu_g&&
|u|^p_{p,g}=\int_M|u|^p d\mu_g\\
|||u|||_{g,\eps}^2=\frac 1{\eps^n}\int_M(\eps^2|\nabla_g u|^2+|u|^2)d\mu_g&&
|u|^p_{p,g,\eps}=\frac 1{\eps^n}\int_M|u|^p d\mu_g
\end{eqnarray*}
\item For $u\in H^1(\R^n)$ we use the norms
\begin{eqnarray*}
||u||^2=\int_{\R^n}(|\nabla u|^2+|u|^2)dx&&
|u|^p_{p}=\int_{\R^n}|u|^p dx
\end{eqnarray*}
\item $m_{\eps,g}=\inf\{J_{\eps,g}(v)\ :\ v\in N_{\eps,g}\}$
\end{itemize}

It is know that there exists a unique positive spherically symmetric function 
$U\in H^1(\R^n)$ such that $J_\infty(U)=m_\infty$. Obviously we have
\begin{equation}
-\Delta U+U=U^{p-1}.
\end{equation}
For $\eps>0$ we set $U_\eps(x)=U(x/\eps)$ and we get 
$-\eps^2\Delta U_\eps+U_\eps=U_\eps^{p-1}$.

Now we shall recall some topological tools which are used in the paper

\begin{definition}\label{poinpol} 
If $(X,Y)$ is a couple of topological spaces, the Poincar\'e polynomial
$P_t(X,Y)$ is defined as the following power series in $t$
\begin{equation}
P_t(X,Y)=\sum_k \dim H_k(X,Y)t^k
\end{equation}
where $H_k(X,Y)$ is the $k$-th homology group of the couple $(X,Y)$ 
with coefficient in some field. 
Moreover we set 
\begin{equation}
P_t(X)=P_t(X,\emptyset)=\sum_k \dim H_k(X)t^k
\end{equation}
\end{definition}
If $X$ is a compact manifold there only a finite number of nontrivial $H_k(X)$ and
$\dim H_k(X)<\infty$. In this case 
$P_t(X)$ is a polynomial and not a formal series.
\begin{definition}\label{mindex}
Let $J$ be a $C^2$ functional on a Banach space $X$ and let $u\in X$ 
be an isolated critical point of $J$ with $J(u)=c$. 
If $J^{c}:=\{v\in X\ :\ J(v)\leq c \}$, then the (polynomial) 
Morse index $i_t(u)$ is the series
\begin{equation}
i_t(u)=\sum_k \dim H_k(J^{c},J^{c}\smallsetminus\{u\})t^k,
\end{equation}
\end{definition}
If $u$ is a nondegenerate critical point of $J$ then 
$i_t(u)=t^{\mu(u)}$ where $\mu(u)$ is the (numerical) Morse index of $u$,
and it is given by the dimension of the maximal subspace on which the 
bilinear form $J''(u)[\cdot,\cdot]$ is negative definite.

It is useful to recall the following result (see \cite{BC94})
\begin{remark}\label{BC}
Let $X$ and $Y$ be topological spaces. If $f:X\to Y$ and $g:Y\to X$ are continuous 
maps such that $g\circ f$ is homotopic to the identity map on $X$, 
then
\begin{equation}
P_t(Y)=P_t(X)+Z(t)
\end{equation}
where $Z(t)$ is a polynomial with non negative coefficients.
\end{remark}
\begin{definition}\label{psdef}
Let $J$ be a $C^1$ functional on a Banach space $X$. 
We say that $J$ satisfies the Palais Smale condition if any sequence 
$\{x_n\}_n\subset X$ for which $J(x_n)$ is bounded and $J'(x_n)\to 0$ 
has a convergent subsequence.
\end{definition}
We now introduce the Banach space $\mathscr{S}^k$ which will be the parameter space. 
We denote by $\mathscr{S}^k$ the Banach space of all $C^k$ symmetric covariant symmetric 
$2$-tensors on $M$. The norm $||\cdot||_k$ is defined in the following way. We fix a 
finite covering $\{V_\alpha\}_{\alpha \in L}$ of $M$ such that the closure of $V_\alpha$ 
is contained in $U_\alpha$ where $\{U_\alpha, \psi_\alpha\}$ is an open coordinate 
neighborhood. If $h\in\mathscr{S}^k$ we denote by $h_{i,j}$ the components of $h$ with 
respect to the coordinates $(x_1,\dots,x_n)$ on $V_\alpha$. We define 
\begin{equation}
||h||_k=\sum_{\alpha\in L}\sum_{|\beta|\leq k}\sum_{i,j=1}^n\sup_{\psi_\alpha(V_\alpha)}
\frac{\partial^\beta h_{i,j}}{\partial x_1^{\beta_1}\cdots \partial x_k^{\beta_k}}
\end{equation}
The set $\mathscr{M}^k$ of all $C^k$ Riemannian metrics on $M$  is an open subset 
of $\mathscr{S}^k$. 

On the tangent bundle of any compact connected Riemannian manifold $M$, 
it is defined the exponential map $\exp: TM\to M$ which is a $C^\infty$ map.
Then, for $\rho $ small enough, the manifold $M$ has a special set of charts given by 
$\exp_x:B(0,R)\to B_g(x,R)$ (where $T_xM$ is identified with $\R^n$) for $x\in M$.
The system of coordinates corresponding to these charts are called normal coordinates.

\begin{remark}\label{rho1}
Let $g_0$ a fixed $C^k$ Riemannian metric on the manifold $M$.
By the compactness of $M$ there exist two positive constant $c,C$ such that
\begin{eqnarray*}
&&\forall x\in M, \forall \xi \in T_xM  \ \ \ 
c||\xi||^2\leq g_0(x)(\xi,\xi)\leq C||\xi||^2;\\
&&\forall x\in M \ \ \ 
c^n\leq |g_0(x)|\leq C^n.
\end{eqnarray*}
\end{remark}

By definition of the norms $|||u|||_{g,\eps}$ and $||h||_k$ we have that  
there exists $\rho_1>0$ such that, if $h\in  \mathscr{B}_{\rho_1}$, the 
two sets $H^1_g(M)$, $H^1_{g_0}(M)$ are the same and the two norms 
$|||u|||_{g,\eps} $, $|||u|||_{{g_0},\eps}$ (as well as 
$|||u|||_{g} $, $|||u|||_{{g_0}}$) are equivalent, and the positive constants for the equivalence 
do not depend on $\eps$, for $0<\eps<1$.

\begin{remark}\label{rho2}
It is trivial that there exists $\rho_2$ such that, for any 
$h\in \mathscr{B}_{\rho_2}$, we have
\begin{equation}
J_{\eps,g_0+h}(1)=\left(\frac 12 -\frac1p\right)\frac1{\eps^n}
\int_M1d\mu_{g_0+h}>\frac{p-2}{2p}\frac 1{\eps^n}\frac  {\mu_{g_0}(M)}2.
\end{equation}
Then $J_{\eps,g_0+h}(1)>2m_\infty$ for 
$\displaystyle \eps<\left[\frac{p-2}{8p m_\infty}\mu_{g_0}(M)\right]^{1/n}$ 
and $h\in \mathscr{B}_{\rho_2}$.
\end{remark}
In the following we consider $g=g_0+h$ with 
$h \in \mathscr{B}_{\hat \rho}$ where $\hat \rho=\min\{\rho_1,\rho_2\}$.
\begin{lemma}\label{lemma3}
There exists $\eps_1\in(0,1)$ such that, for any $\eps<\eps_1$ and for any 
$h\in \mathscr{B}_{\hat \rho}$, we have
\begin{equation}
A(h,\eps):=\{u\in N_{\eps, g_0+h}\ :\ J_{\eps,g_0+h}(u)\leq2m_\infty\}
\subset I(0,\alpha)\smallsetminus 1
\end{equation}
for some $\alpha>0$.
\end{lemma}
\begin{proof}
If $u\in A(h,\eps)$ we have 
\begin{equation}\label{eq2-1}
J_{\eps,g}(u)= \left(\frac 12 -\frac 1p\right) |||u|||_{\eps,g}^2\leq 2m_\infty,
\end{equation}
where $g=g_0+h$ with $h\in\mathscr{B}_{\hat \rho}$. By definition of 
$\hat \rho$, there exists 
$c_1>0$ such that 
\begin{equation}\label{eq2-2}
c_1||u||^2_{g_0}\leq c_1|||u|||^2_{\eps_1,g_0}\leq  c_1|||u|||^2_{\eps,g_0}\leq
|||u|||^2_{\eps,g}
\end{equation}
for $0<\eps<\eps_1$ and $h\in \mathscr{B}_{\hat \rho}$. 
Then by (\ref{eq2-1}) and (\ref{eq2-2}) we have 
\begin{equation}
||u||^2_{g_0}\leq \frac {2p}{p-2}\frac{2m_\infty}{c_1}.
\end{equation}
By Remark \ref{rho2} we have $u\neq 1$ for $\eps_1$ small enough.
By the following Remark \ref{remark5-9}, 
we have that $A(h,\eps)\neq \emptyset$ for $\eps_1$, and 
$\hat \rho$ small enough. 
\end{proof}
Now we recall a result about the nondegeneracy of positive solutions of (\ref{eq1-1}) 
with respect to the pair of parameters $(\eps,g)$, where $\eps$ is a positive number and 
$g$ is a Riemannian metric (see \cite{MP09c}).
\begin{theorem}\label{teogen}
Given $g_0\in\mathscr{M}^k$, and an open ball of $H^1_{g_0}(M)$ without the constant $1$ 
$A=I(0,\alpha)\smallsetminus \{1\}$, the set 
\begin{equation*}
D=\left\{
\begin{array}{c} 
(\eps,h)\in(0,1)\times\mathscr{B}_{\rho}
\text{ s.t. any }u\in A 
\text{ solution of the equation }\\-\eps^2\Delta_{g_0+h}u+u=(u^+)^{p-1}
\text{is nondegenerate}
\end{array}
\right\}
\end{equation*}
is an residual subset in $(0,1)\times\mathscr{B}_{\rho}$ for 
$\rho$ small enough.
\end{theorem}

\section{The main ingredient of the proof}\label{sec3}

Let us sketch the proof of our main result. We are going to find an estimate of the number 
of nonconstant critical points of the functional $J_{\eps,g_0+h}$ with energy close to 
$m_\infty$, with respect to the parameters $(\eps,h)\in (0,\tilde\eps)\times \mathscr{B}_{\tilde\rho}$.

First of all we apply Theorem \ref{teogen} choosing the positive numbers $\tilde\eps,\tilde\rho$ 
small enough and the open bounded set $A$ equal to $I(0,\alpha)\smallsetminus \{1\}$, 
where $\alpha$ is given by Lemma \ref{lemma3}. So we get that the set 
\begin{equation*}
\begin{split}
&D(\tilde\eps,\tilde\rho)=\left\{
\begin{array}{l} 
(\eps,h)\in(0,\tilde\eps)\times\mathscr{B}_{\tilde\rho}:
\text{ any }u\in I(0,\alpha)\smallsetminus\{1\} 
\text{ solution of}\\
-\eps^2\Delta_{g_0+h}u+u=(u^+)^{p-1}
\text{ is non degenerate } 
\end{array}
\right\}\supset\\
&\supset
\left\{
\begin{array}{l} 
(\eps,h)\in(0,\tilde\eps)\times\mathscr{B}_{\tilde\rho}:
\text{any solutions of }
-\eps^2\Delta_{g_0+h}u+u=(u^+)^{p-1}\\
\text{ nonconstant, such that } J_{\eps,g_0+h}(u)<2m_\infty
\text{ is non degenerate } 
\end{array}
\right\}
\end{split}
\end{equation*}
is an residual subset in $(0,\tilde\eps)\times\mathscr{B}_{\tilde\rho}$, for $\tilde\eps$ and 
$\tilde \rho$ small enough.

Since $\displaystyle \lim_{(\eps,h)\to(0,0)}m_{\eps,g_0+h}=m_\infty$ (see following 
Remark \ref{remark5-9}), given $\delta\in(0,m_\infty/4)$, for $(\eps,h)\in \R^+\times\mathscr{S}^k$ 
small enough, we have $$0<m_\infty-\delta<m_{\eps,g_0+h}<m_\infty+\delta<2m_\infty.$$ 
Thus $m_\infty-\delta$ is not a critical value of $J_{\eps,g_0+h}$. 

By 
the compactness of $M$, it holds Palais Smale condition for the functional $J_{\eps,g_0+h}$. 
At this point we take $(\eps,h)\in D(\tilde\eps,\tilde\rho)$ with the positive numbers 
$\tilde\eps$, $\tilde\rho$ small enough. Thus we have that the critical points $u$ of 
$J_{\eps,g_0+h}$ with $J_{\eps,g_0+h}(u)<2m_\infty$ are in a finite number, then we can assume 
that $m_\infty+\delta$ is not a critical value for $J_{\eps,g_0+h}$. 
It holds the following relation proved in $\cite{Ben95, BC94}$ 
(see \cite[Lemma 5.2]{BC94})
\begin{equation}\label{eq3-1}
P_t( J_{\eps,g_0+h}^{m_\infty+\delta} ,J_{\eps,g_0+h}^{m_\infty-\delta})=
tP_t( J_{\eps,g_0+h}^{m_\infty+\delta}\cap  N_{\eps,g_0+h}).
\end{equation}
 
On the other hand, by proposition 
\ref{prop4-2} and \ref{prop5-10} and Lemma \ref{lemma5-7}, we can build two maps 
$\Phi_{\eps,g_0+h}$ and $\beta_{g_0+h}$ such that 
\begin{equation}
M\stackrel{\Phi_{\eps,g_0+h}}{\longrightarrow}
N_{\eps, g_0+h}\cap J_{\eps,g_0+h}^{m_\infty+\delta}
\stackrel{\beta_{g_0+h}}{\longrightarrow} M_{r(M)}, 
\end{equation}
where $\beta_{g_0+h}\circ\Phi_{\eps,g_0+h}$ is homotopic to the
identity map and $M_{r(M)}$ is homotopically equivalent to $M$. 
Therefore, by Remark \ref{BC} we have 
\begin{equation}\label{eq3-2}
P_t( J_{\eps,g_0+h}^{m_\infty+\delta}\cap  N_{\eps,g_0+h})=P_t(M)+Z(t)
\end{equation}
where $Z(t)$ is a polynomial with nonnegative integer coefficients.

Since the functional $J_{\eps,g_0+h}$ satisfies the Palais Smale condition and 
the critical points $u$ of $J_{\eps,g_0+h}$ 
such that  $J_{\eps,g_0+h}(u)<m_\infty+\delta$ 
are nondegenerate, by Morse theory we have
\begin{equation}\label{eq3-4}
\sum\limits_{u\in C}i_t(u)=
\sum\limits_{u\in C}t^{\mu(u)}=
P_t( J_{\eps,g_0+h}^{m_\infty+\delta},J_{\eps,g_0+h}^{m_\infty-\delta}).
\end{equation}
Here $\mu(u)$ is the dimension of the maximal subspace on which the bilinear form 
$J_{\eps,g_0+h}''(u)[\cdot,\cdot]$ is negative definite and the set $C$ is defined by 
\begin{equation}\label{eq3-3}
C=\{u\ :\ J'_{\eps,g_0+h}(u)=0\text{ and }
m_\infty-\delta <J_{\eps,g_0+h}(u)<m_\infty+\delta \}.
\end{equation}
Then by (\ref{eq3-1}), (\ref{eq3-4}) and (\ref{eq3-3}), for any $(\eps,h)\in D(\tilde\eps,\tilde\rho)$
with the positive numbers $\tilde\eps$, $\tilde\rho$ small enough, we get that the functional 
$J_{\eps,g_0+h}$ has at least $P_1(M)$ nonconstant critical points $u$ such that 
$J_{\eps,g_0+h}(u)<m_\infty+\delta<2m_\infty$.

\section{The function $\Phi_{\eps,g}$}\label{sec4}

Let us define a smooth real cut off function $\chi_R$ such that 
$\chi_R(t)=1$ if $0\leq t\leq R/2$,
$\chi_R(t)=0$ if $t\geq R$, and $|\chi'(t)|\leq 2/R$. Fixed $q\in M$ and $\eps>0$, we define 
on $M$ the function 
\begin{equation}\label{eq4-1}
w^g_{q,\eps}(x)=\left\{
\begin{array}{cc}
U_\eps(\exp_q^{-1}(x))\chi_R(|\exp_q^{-1}(x)|)&\text{if }x\in B_g(q,R);\\
0&\text{otherwise}.
\end{array}
\right.
\end{equation}
For any $u\in H^1_{g_0+h}(M)$ with $u^+\neq 0$ we define $t(u)\in \R$ as 
\begin{equation}
t^{p-2}(u)=\frac{\displaystyle\int_M (\eps^2|\nabla_{g_0+h}u|^2+u^2)d\mu_{g_0+h}}
{\displaystyle\int_M |u^+|^pd\mu_{g_0+h}},
\end{equation}
so $t(u)$ is the unique number such that $t(u)u\in N_{\eps,g_0+h}$.

Thus we can define a map $\Phi_{\eps,g}:M\to N_{\eps,g}$ by 
\begin{equation}\label{phieps}
\Phi_{\eps,g}(q)=t(w^g_{q,\eps})w^g_{q,\eps}.
\end{equation}
\begin{proposition}\label{prop4-2}
Given $g_0\in\mathscr{M}^k$, for any $\eps>0$ and for any 
$h\in \mathscr{B}_{\tilde\rho}\subset\mathscr{S}^k$ the operator 
$\Phi_{\eps,g_0+h}:M\to N_{\eps,g_0+h}$ is continuous. Moreover, given $g_0$, 
for any $\delta>0$ there exists $\eps_2=\eps_2(\delta)$ such that,  
if $\eps<\eps_2$, then 
\begin{eqnarray}
\Phi_{\eps,g_0+h}(q)\in N_{\eps,g_0+h}\cap J_{\eps,g_0+h}^{m_\infty+\delta} &&
\forall q\in M, \forall h\in \mathscr{B}_{\hat \rho}.
\end{eqnarray}
\end{proposition}
\begin{proof}
It easy to prove the continuity of $q\to \Phi_{\eps,g_0+h}(q)\in N_{\eps,g_0+h}$ 
from $M$ to $H^1_{g_0+h}(M)$. 
To obtain the second statement we recall that
\begin{equation}
J_{\eps,g_0+h}(\Phi_{\eps,g_0+h}(w_{q,\eps}^{g_0+h}))=
\frac1{\eps^n}\left(\frac 12- \frac 1p\right)
[t(w_{q,\eps}^{g_0+h})]^p|w_{q,\eps}^{g_0+h}|^p_{p,g_0+h}.
\end{equation}
Moreover the following limits hold
\begin{eqnarray}
&&\lim_{\eps\to0} \frac 1{\eps^n}|w_{q,\eps}^{g_0+h}|^2_{2,g_0+h}=|U|^2_2\\
&&\lim_{\eps\to0} \frac 1{\eps^n}|w_{q,\eps}^{g_0+h}|^p_{p,g_0+h}=|U|^p_p\\
&&\lim_{\eps\to0} \frac {\eps^2}{\eps^n}|\nabla_{g_0+h}w_{q,\eps}^{g_0+h}|^2_{2,g_0+h}=
|\nabla U|^2_2
\end{eqnarray}
uniformly with respect to $q\in M$ and $h\in \mathscr{B}_\rho\subset \mathscr{S}^k$, $k\geq 2$.
We prove only the first limit, the others follow in a similar way. Using the normal coordinates
with respect to $g_0+h$ at the point $q\in M$ we have
\begin{equation}\label{eq29}
\begin{split}
\frac 1{\eps^n}|w_{q,\eps}^{g_0+h}|^2_{2,g_0+h}&-|U|^2_2=
\left|\int_{\R^n}U^2(z)[|\chi^2(\eps|z|)|\  |g_{0,q}(\eps z)+h_q(\eps z)|^{1/2}-1]dz\right|
\leq\\
\leq&\left|\int_{B(0,T)}U^2(z)[|\chi^2(\eps|z|)|\ |g_{0,q}(\eps z)+h_q(\eps z)|^{1/2}-1]dz\right|+\\
&+\left|
\int_{\R^n\smallsetminus B(0,T)}
U^2(z)[|\chi^2(\eps|z|)|\ |g_{0,q}(\eps z)+h_q(\eps z)|^{1/2}-1]dz\right|.
\end{split}
\end{equation}
We point out that
\begin{equation*}
\begin{split}
(g_{0,q}(0)+h_q(0))^{ij}&=\delta_{ij} \text{ and }\\
&\\
|(g_{0,q}(y)+h_q(y))^{ij}-\delta_{ij}|&=
|\nabla[(g_{0,q}({\theta y}))^{ij}]\cdot y
+\nabla[(h_q(\theta y))^{ij}]\cdot y|\leq\\
&\leq[|\nabla(g_{0,q}({\theta y}))^{ij}|
+|\nabla(h_q(\theta y))^{ij}|]\cdot |y|.
\end{split}
\end{equation*}
Because $M$ is compact, we have
$[|\nabla(g_{0,q}(\theta y))^{ij}|
+|\nabla(h_q(\theta y))^{ij}|]$ is bounded independently on $y\in B(0,r)$, 
$q\in M$, and $h\in \mathscr{B}_\rho$.

At this point it is clear that the second addendum of formula (\ref{eq29}) vanishes 
as $T\to+\infty$. Moreover, fixed $T$ large enough, the first addendum of
(\ref{eq29}) vanishes as $\eps \to0$.

By the previous limits we get that $\displaystyle \lim_{\eps\to0}t(w_{\eps,q}^{g_0+h})=1$ 
uniformly with respect to $q\in M$ and $h\in \mathscr{B}_\rho$. Then we have 
\begin{equation}
\lim_{\eps\to0}J_{\eps,g_0+h}\big(t(w_{\eps,q}^{g_0+h})w_{\eps,q}^{g_0+h}\big)=m_\infty
\end{equation}
uniformly with respect to $q\in M$ and $h\in \mathscr{B}_\rho$. 
\end{proof}

\begin{remark}\label{rem4-3}
By Proposition \ref{prop4-2} we get 
\begin{equation}\label{eq4-4}
\limsup_{\eps\to0} m_{\eps,g_0+h}\leq m_\infty,
\end{equation}
uniformly with respect to $h\in \mathscr{B}_\rho$. 
Here $\displaystyle m_{\eps,g}=\inf_{N_{\eps,g}}J_{\eps,g}$ and 
$\displaystyle m_\infty=\inf_{N_\infty}J_\infty$.
\end{remark}

\section{The operator $\beta_g$}\label{sec5}
For any function $u\in N_{\eps,g}$ we can define its centre of mass 
as a point $\beta_g(u)\in \R^N$ by 
\begin{equation}
\beta_g(u)=\frac{\displaystyle \int_Mx(u^+)^pd\mu_g}{\displaystyle \int_M(u^+)^pd\mu_g}.
\end{equation}
The function $\beta_g$ is well defined on $N_{\eps,g}$ since, if $u\in N_{\eps,g}$ then 
$u^+\neq0$. We will prove that, if $u\in N_{\eps,g}\cap J_{\eps,g}^{m_\infty+\delta}$,
then $\beta_g(u)\in M_{r(M)}$, using the concentration properties of the functions in 
$N_{\eps,g}\cap J_{\eps,g}^{m_\infty+\delta}$ as $\eps$ and $\delta$ are suitably small.
In the following we use the same arguments of Section 5 of [BBM], but here we have to take 
in account the dependence of the metric $g$, hence some calculations are different.

Our aim is to get the following statement
\begin{proposition}\label{prop5-10}
Given $g_0\in \mathscr{M}^k$, 
there exist $\delta_0$, $\rho_0$ and $\eps_0$ such that, for any $\delta\in(0,\delta_0)$,
for any $\eps\in(0,\eps_0)$, for any $h\in \mathscr{B}_{\rho_0}$, and for any 
$u\in N_{\eps,g_0+h}\cap J_{\eps,g_0+h}^{m_\infty+\delta}$,
it holds $\beta_{g_0+h}(u)\in M_{r(M)}$ where $r(M)$ is the radius of topological invariance of $M$
and $M_{r(M)}=\{x\in \R^N\ :\ d(x,M)<r(M)\}$.
\end{proposition}

To prove Proposition \ref{prop5-10} we need some technical results. 

First of all we consider partitions of the compact manifold $M$. Given $\eps>0$ and a metric 
$g_0$, a partition $\mathscr{P}=\mathscr{P}(\eps)=\{P_j=P_J(\eps)\}_{j\in\Lambda(\eps)}$
is called a ``good'' partition if:
\begin{itemize}
\item for any $j\in\Lambda(\eps)$ the set $P_j$ is closed
\item $P_j\cap P_i\subset \partial  P_j\cap \partial P_i$ for $i\neq j$
\item there exists $\rho>0$ such that, for any $j$, there exists $q_j\in P_j$ such that 
$B_g(q_j,\eps)	\subset P_j\subset B_g(q_j,(1+1/a)\eps)$ with a constant $a$ 
independent on $\eps$ and $g=g_0+h$ with $h\in \mathscr{B}_\rho$
\item any point $x\in M$ is contained in at most $\nu_M$ balls 
$B_g(q_j,(1+1/a)\eps)$ where $\nu_M$ does not depend on $\eps$ and $g$, 
$g=g_0+h$ and $h\in \mathscr{B}_\rho$.
\end{itemize}

\begin{lemma}\label{lemma5-3}
There exist $\gamma>0$ and $\rho>0$ such that, for any $\delta>0$ and $\eps>0$, given any 
``good'' partition $\mathscr{P}(\eps)$ 
and any $u\in N_{\eps,g}\cap J_{\eps,g}^{m_\infty+\delta}$, 
where $g=g_0+h$ with $h\in \mathscr{B}_{\rho}$, there exists a set 
$P\in \mathscr{P}(\eps)$ such that 
\begin{equation}
\frac 1{\eps^n}\int_P(u^+)^pd\mu_{g_0+h}\geq\gamma, \text{ with }h\in \mathscr{B}_{\rho}
\end{equation}
\end{lemma}

The proof of this Lemma can be obtained following the same argument of the proof of 
\cite[Lemma 5.3]{BBM07}. Indeed, every constant appearing in 
\cite[Lemma 5.3]{BBM07} can be chosen independently on $h\in \mathscr{B}_{\rho}$ with 
$\rho$ small enough.

\begin{proposition}\label{prop5-5}
For any $\eta\in(0,1)$ there exist $\delta_0$, $\rho_0$ and $\eps_0$ such that, for any $\delta<\delta_0$, 
for any $\eps<\eps_0$, for any $h\in\mathscr{B}_{\rho_0} $ and 
for any $u\in N_{\eps,g_0+h}\cap J_{\eps,g_0+h}^{m_\infty+\delta}$, there exists $q=q(u)$ for that
\begin{equation}
\left(\frac12-\frac1p\right)\frac1{\eps^n}
\int_{B_{g_0+h}(q,r(M)/2)}(u^+)^pd\mu_{g_0+h}>(1-\eta)m_\infty
\end{equation}
\end{proposition}
\begin{proof}
We only prove the proposition for any $u \in N_{\eps,g_0+h}\cap 
J_{\eps,g_0+h}^{m_{\eps,g_0+h}+2\delta}$.
Indeed, by this result and by Remark \ref{rem4-3} we get  
\begin{equation}
\lim_{(\eps,h)\to(0,0)}m_{\eps,g_0+h}=m_\infty.
\end{equation}
Hence it holds
$J_{\eps,g_0+h}^{m_\infty+\delta}\subset J_{\eps,g_0+h}^{m_{\eps,g_0+h}+2\delta}$ for 
$\delta$, $\eps$ and $||h||_k$ small enough. So the thesis holds. 

We argue by contradiction. Suppose that there exists $\eta\in (0,1)$ such that we can find 
sequences of vanishing numbers $\{\delta_k\}_k$, 
 $\{\eps_k\}_k$, a sequence $h_k\to0$ in  $\mathscr{S}^k$
and  a sequence $u_k\in  N_{\eps_k,g_0+h_k}\cap J_{\eps_k,g_0+h_k}^{m_{\eps,g_0+h_k}+2\delta_k}$
such that, for all $q\in M$, 
\begin{equation}\label{eq5-8}
\left(\frac12-\frac1p\right)\frac1{\eps_k^n}
\int_{{B}_{g_0+h_k}(q,r(M)/2)}(u_k^+)^pd\mu_{g_0+h_k}\leq(1-\eta)m_\infty.
\end{equation}
By Ekeland variational principle (see \cite{deF89}), 
and by definition of the Nehari manifold, we can assume that
\begin{equation}\label{eq5-9}
|J'_{\eps_k,g_0+h_k}(u_k)(\xi)|\leq\sqrt{\delta_k}|||\xi|||_{\eps_k}\ \ \forall \xi\in H^1_{g_0+h_k}(M).
\end{equation}
By Lemma \ref{lemma5-3} there exists a set $P_k\in \mathscr{P}_{\eps_k}$ such that 
\begin{equation}\label{eq5-10}
\frac 1{\eps_k^n}\int_{P_k}(u_k^+)^pd\mu_{g_0+h_k}\geq\gamma.
\end{equation}
We choose a point $q_k$ interior to $P_k$ and we consider the function 
$w_k:\R^n\to\R$ defined by 
\begin{equation}
u_k(x)\chi_R(\exp_{q_k}^{-1}(x))=u_k(\exp_{q_k}(\eps_kz)\chi_R(\eps_k|z|)=w_k(z)
\end{equation}
where $x\in M$ and $\chi_R$ is a smooth cut off function $\chi_R(t)\equiv1$ 
for $0<t<R/2$, $\chi_R(t)\equiv0$ for $t>R$ and $R$ small enough. It easily follows that 
$w_k\in H^1_0(B(0,R/\eps_k))\subset H^1(\R^n)$.

We now establish some properties of the functions $w_k$ by some lemmas. The proof of these lemmas 
are in Section \ref{sei}

\begin{lemma}\label{lemma5-6}
By considering a subsequence, there exists $w\in H^1(\R^n)$ such that $\lim_k w_k=w$ as a 
weak limit in $H^1(\R^n)$ and a strong limit in $L^p_{\text{loc}}(\R^n)$.
\end{lemma}

\begin{lemma}\label{lemma5-7}
The limit function $w\in H^1(\R^n)$ is a weak solution of
\begin{equation}
-\Delta w+w=|w|^{p-2}w,\ \ w>0
\end{equation}
\end{lemma}

At this point we observe that, by definition of $w_k$ and by (\ref{eq5-8}), 
for any $\sigma\in(0,\eta)$, and for any $T>0$ we have, for $k$ large enough,
\begin{displaymath}
\left(\frac12-\frac1p\right)
\int_{B(0,T)}|w_k^+(x)|^pdx<\frac{1-\eta}{1-\sigma}m_\infty.
\end{displaymath}
On the other hand by Lemma \ref{lemma5-7} we have 
\begin{displaymath}
\left(\frac12-\frac1p\right)
||w||^2=\left(\frac12-\frac1p\right)|w|^p_p\geq m_\infty.
\end{displaymath}
By Lemma \ref{lemma5-6}, for $T$ and $k$ large enough
\begin{displaymath}
\left(\frac12-\frac1p\right)
\int_{B(0,T)}|w_k^+(x)|^pdx>\frac{1-\eta}{1-\sigma}m_\infty,
\end{displaymath}
and this leads to a contradiction.
\end{proof}

\begin{remark}\label{remark5-9}
By Proposition \ref{prop5-5} and by Remark \ref{rem4-3}, it holds
$$\displaystyle\lim_{(\eps,h)\to(0,0)}m_{\eps,g_0+h}=m_\infty$$ uniformly with respect to 
$h\in\mathscr{B}_\rho$ (here $\rho$ is given as in Proposition \ref{prop5-5})
\end{remark}

\begin{proof}[Proof of Proposition \ref{prop5-10}]
By Proposition \ref{prop5-5} for any $\eta\in(0,1)$ and for any 
$u\in N_{\eps,g_0+h}\cap J_{\eps,g_0+h}^{m_\infty+\delta}$ with $\eps$, $\delta$, and $h$ 
small enough, there exists $q\in M$ such that 
\begin{equation}\label{eq48}
(1-\eta)m_\infty<\left(\frac 12 -\frac 1p \right)\frac 1{\eps^n}
\int_{B_{g_0+h}(q,r(M)/2)}(u^+)^pd\mu_{g_0+h}.
\end{equation}
Moreover, since $u \in N_{\eps,g_0+h}\cap J_{\eps,g_0+h}^{m_\infty+\delta}$, it holds
\begin{equation}\label{eq49}
\left(\frac 12 -\frac 1p \right)\frac 1{\eps^n}
\int_M(u^+)^pd\mu_{g_0+h}<m_\infty+\delta.
\end{equation}
Then, by (\ref{eq48}) and (\ref{eq49}), we have
\begin{equation}
\begin{split}
|\beta_g(u)-q|\leq& 
\left|\frac{\int_M (x-q)(u^+(x))^pd\mu_{g_0+h}}{\int_M (u^+(x))^pd\mu_{g_0+h}}\right|\leq\\
\leq&
\left|\frac{\int_{B_{g_0+h}(q,r(M)/2)} (x-q)(u^+(x))^pd\mu_{g_0+h}}{\int_M (u^+(x))^pd\mu_{g_0+h}}\right|+\\
&+
\left|\frac{\int_{M\smallsetminus B_{g_0+h}(q,r(M)/2)} (x-q)(u^+(x))^pd\mu_{g_0+h}}
{\int_M (u^+(x))^pd\mu_{g_0+h}}\right|\leq\\
\leq& \frac{r(M)}2 +2D\left(1-\frac{(1-\eta)m_\infty}{m_\infty+\delta}\right)
\end{split}
\end{equation}
where $D$ is the diameter of $M$ as a compact subset of $\R^N$. Choosing $\eta$ and $\delta$ 
suitably small we get the claim.
\end{proof}

The last result of this section is that the composition $I^g_\eps:=\beta_g\circ\Phi_{\eps,g}$ 
is homotopic to the identity on $M$.
\begin{lemma}\label{lemma5-11}
There exists $\eps_2<\eps_0$ such that, for any $\eps\in(0,\eps_2)$ and for any 
$h\in \mathscr{B}_{\hat\rho}$ 
\begin{equation}
I^{g_0+h}_\eps:=\beta_{g_0+h}\circ\Phi_{\eps,g_0+h}:M\to M_{r(M)}
\end{equation}
is well defined and it is homotopic to the identity on $M$.
\end{lemma}
\begin{proof}
By Proposition \ref{prop4-2} and (\ref{eq5-10}) the map $I^{g_0+h}_\eps$ is well defined. 
To prove that $I^{g_0+h}_\eps$ is  homotopic to the identity on $M$ it is enough to evaluate 
the map. Here $g=g_0+h$.
\begin{equation}
I^g_\eps(q)-q=\frac{\eps\int_{B(0,R/\eps)} z|U(z)\chi_R(\eps|z|)|^p|g_q(\eps z)|^{\frac 12}dz}
{\int_{B(0,R/\eps)} |U(z)\chi_R(\eps|z|)|^p|g_q(\eps z)|^{\frac 12}dz},
\end{equation}
hence $| I^g_\eps(q)-q|\leq c\cdot\eps$ for a constant $c=c(M,g_0,\hat\rho)$ 
that does not depend on $q$ and 
on $h\in\mathscr{B}_{\hat\rho}$
\end{proof}

\section{Proof of technical lemmas}\label{sei}

\begin{proof}[Proof of Lemma \ref{lemma5-6}]
Here $g_k=g_0+h_k$ with $h_k\in \mathscr{B}_\rho$, $h_k\to 0$ and 
$\tilde u_k(y) =u_k(\exp_{q_k}(y))$. 
We recall that $\displaystyle \left(\frac 12-\frac 1p\right)|||u_k|||^2_{\eps_k,g_k}
\leq m_{\eps_k,g_k}+2\delta_k$.
By Remark \ref{rem4-3}, 
for $k$ large we get 
\begin{equation*}
\begin{split}
2\frac{2p}{p-2}&m_\infty  \geq \frac{1}{\eps^n_k}\int_Mu_k^2(x)d\mu_{g_k}\geq
\frac{1}{\eps^n_k} \int_{B_{g_k}(q_k,R)}\chi_R^2(|\exp_{q_k}^{-1}(x)|)u_k^2(x)d\mu_{g_k}=\\
&= \frac{1}{\eps^n_k} \int_{B(0,R)}\chi_R^2(|y|)u_k^2(\exp_{q_k}y)|g_{k,q_k}(y)|^{1/2}dy=\\
&=\int_{B(0,R/{\eps_k})}\chi_R^2(\eps_k|z_k|)\tilde u_k^2(\eps_kz)|g_{k,q_k}(\eps_kz)|^{1/2}dz\geq
\text{const}\int_{\R^n}w_k^2(z)dz.
\end{split}
\end{equation*}
Let us now estimate the $L^2$ norm for $\nabla w_k$
\begin{equation*}
\int_{B(0,R/\eps_k)} \sum_i\left(\frac{\partial w_k}{\partial z_i}(z)\right)^2dz=I_1+I_2+I_3,
\text{ where}
\end{equation*}
\begin{equation*}
\begin{split}
I_1&=\int_{B(0,R/\eps_k)}\chi_R^2(\eps_k|z|)\sum_i
\left(\frac{\partial \tilde u_k}{\partial z_i}(\eps_kz)\right)^2dz
\\
I_2&=\int_{B(0,R/\eps_k)}\tilde u_k^2(\eps_k|z|)\sum_i
\left(\frac{\partial \chi_R}{\partial z_i}(\eps_kz)\right)^2dz\\
I_3&=2\int_{B(0,R/\eps_k)}\chi_R(\eps_k|z|)\tilde u_k(\eps_kz)
\sum_i\frac{\partial \chi_R}{\partial z_i}(\eps_k|z|)
\frac{\partial \tilde u_k}{\partial z_i}(\eps_k|z|)
\end{split}
\end{equation*}
By Remark \ref{rho1} we have 
\begin{equation}
\begin{split}
2\frac{2p}{p-2}m_\infty&\geq
\frac{\eps_k^2}{\eps_k^n}\int_M |\nabla_{g_k}u_k(x)|^2d\mu_{g_k}\geq\\
&\geq\int_{B(0,R/\eps_k)}\left(
\sum_{ij}g_{k,q_k}^{ij}(\eps_kz)\frac{\partial \tilde u_k}{\partial z_i}(\eps_kz)
\frac{\partial \tilde u_k}{\partial z_j}(\eps_kz)\right)|g_{k,q_k}(\eps_kz)|^{1/2}dz\geq\\
&\geq\text{const} \int_{B(0,R/\eps_k)}|\nabla \tilde u_k(\eps_kz)|^2dz.
\end{split}
\end{equation}
Thus $I_1$ is bounded. Analogously for addenda $I_2$ and $I_3$
\end{proof}

\begin{proof}[Proof of Lemma \ref{lemma5-7}]
For any $\varphi\in C^\infty_0(\R^n)$ we have that, for $k$ big enough, 
$w_k(z)=u_k(\exp_{q_k}^{-1}(\eps_kz))$ for $z\in\supt \varphi$, because 
$\supt \varphi\subset\{z\ :\ \chi_R(\eps_k|z|)=1\}$ for $k$ big enough. 
We put $\displaystyle \hat\varphi_k(x)=\varphi_k\left(\frac 1{\eps_k}\exp_{q_k}^{-1}(x)\right)$
for $x\in M$. If $\supt \varphi\subset B(0,T)$, then 
$\supt \hat\varphi_k\subset B_{g_k}(q_k,\eps_kT)$ so we have 
\begin{equation}\label{5-7-1}
\begin{split}
J'_{\eps_k}&(u_k)[\hat\varphi_k]=\frac1{\eps^n}\int_M
\eps_k^2\nabla_{g_k}u_k\nabla_{g_k}\hat \varphi_k+u_k\hat\varphi_k-
(u_k^+)^{p-1}\hat\varphi_kd\mu_{g_k}=\\
&=\int_{B(0,T)}|g_{k,q_k}(\eps_kz)|^{1/2}
\left[\sum_{ij}g^{ij}_{k,q_k}(\eps_kz)\frac{\partial w_k}{\partial z_i}
\frac{\partial \hat\varphi_k}{\partial z_j}+w_k\hat\varphi_k-(w_k^+)^{p-1}\hat\varphi_kdz.
\right]
\end{split}
\end{equation}
By the Ekeland principle we have 
\begin{equation}
|J'_{\eps_k}(u_k)[\hat\varphi_k]|\leq\sqrt{\delta_k}|||\hat\varphi_k|||_{\eps_k}.
\end{equation}
It is sufficient to prove that $|||\hat\varphi_k|||_{\eps_k}$  is bounded to obtain that 
\begin{equation}\label{5-7-2}
J'_{\eps_k}(u_k)[\hat\varphi_k]\to0\ \ \text{ as }k\to\infty.
\end{equation}
In fact we have
\begin{equation}\label{5-7-3}
|||\hat\varphi_k|||^2_{\eps_k}=
\int_{B(0,T)}\left[\sum_{ij}g^{ij}_{k,q_k}(\eps_kz)\frac{\partial\hat\varphi_k}{\partial z_i}
\frac{\partial\hat\varphi_k}{\partial z_j}+\hat\varphi_k^2
\right]|g_{k,q_k}(\eps_kz)|^{1/2}dz
\end{equation}
and 
\begin{equation}\label{5-7-4}
\begin{split}
g^{ij}_{k,q_k}(\eps_kz)&=g^{ij}_{0,q_k}(\eps_kz)+h_k^{ij}(\eps_kz)=\\
&=\delta_{ij}+\eps_{k}dg_{0,q_k}(\theta \eps_kz)(z)+h_k^{ij}(\eps_kz).
\end{split}
\end{equation}
Because $h_k\to0$ as $k\to\infty$ in the Banach space $\mathscr{S}^k$, by (\ref{5-7-2}) and
(\ref{5-7-3}) we get the boundness of $|||\hat\varphi_k|||_{\eps_k}^2$.

By (\ref{5-7-1}) and (\ref{5-7-4}) we have
\begin{equation}\label{5-7-5}
J'_{\eps_k}(u_k)[\hat\varphi_k]\to J'_\infty(w)[\varphi]\ \ \forall\varphi\in C^\infty_0(\R^n).
\end{equation}
So by (\ref{5-7-5}) and (\ref{5-7-2}) we have $-\Delta w+w=|w|^{p-2}w$ with $w\geq0$. 
Now we show that $w\neq0$. By the definition of ``good'' partition, by Lemma \ref{lemma5-3} 
and by (\ref{eq5-10}) we can choose a number $T>0$ and $q_k\in M$ such that, for $k$ big enough
$q_k\in P_k\subset B_{g_k}(q_k,\eps_kT)$. By Remark \ref{rho1} 
and Lemma \ref{lemma5-3} we 
get, for $k$ large enough
\begin{equation}
\begin{split}
\int_{B(0,T)}(w_k^+)^p&=
\int_{B(0,T)}(u_k^+(\exp_{q_k}(\eps_kz)))^p\frac{|g_{k,q_k}(\eps_kz)|^{1/2}}{|g_{k,q_k}(\eps_kz)|^{1/2}}dz
\geq\\
&\geq\text{const}\frac1{\eps^n}\int_{B_{g_k}(q_k,\eps_kT)}|u^+(x)|^pd\mu_g\geq\text{const}\cdot\gamma.
\end{split}
\end{equation}
So we get $w\neq0$ because $w_k$ converges strongly to $w\in L^p(B(0,T))$ 
by Lemma \ref{lemma5-6}, hence $w_k^+$ converge strongly to $w^+=w$ in $L^p(B(0,T))$.
\end{proof}


\begin{thebibliography}{10}

\bibitem{BaCo88}
A.~Bahri and J.-M. Coron, \emph{On a nonlinear elliptic equation involving the
  critical {S}obolev exponent: the effect of the topology of the domain}, Comm.
  Pure Appl. Math. \textbf{41} (1988), no.~3, 253--294.

\bibitem{BaL90}
Abbas Bahri and Yan~Yan Li, \emph{On a min-max procedure for the existence of a
  positive solution for certain scalar field equations in {${\bf R}\sp N$}},
  Rev. Mat. Iberoamericana \textbf{6} (1990), no.~1-2, 1--15.

\bibitem{BaLi97}
Abbas Bahri and Pierre-Louis Lions, \emph{On the existence of a positive
  solution of semilinear elliptic equations in unbounded domains}, Ann. Inst.
  H. Poincar\'e Anal. Non Lin\'eaire \textbf{14} (1997), no.~3, 365--413.

\bibitem{Ben95}
Vieri Benci, \emph{Introduction to {M}orse theory. {A} new approach},
  Topological Nonlinear Analysis: Degree, Singularity, and Variations (Michele
  Matzeu and Alfonso Vignoli, eds.), Progress in Nonlinear Differential
  Equations and their Applications, no.~15, Birkh{\"a}user, Boston, 1995,
  pp.~37--177.

\bibitem{BBM07}
Vieri Benci, Claudio Bonanno, and Anna~Maria Micheletti, \emph{On the
  multiplicity of solutions of a nonlinear elliptic problem on {R}iemannian
  manifolds}, J. Funct. Anal. \textbf{252} (2007), no.~2, 464--489.

\bibitem{BC91}
Vieri Benci and Giovanna Cerami, \emph{The effect of the domain topology on the
  number of positive solutions of nonlinear elliptic problems}, Arch. Rational
  Mech. Anal. \textbf{114} (1991), no.~1, 79--93.

\bibitem{BC94}
Vieri Benci and Giovanna Cerami, \emph{Multiple positive solutions of some elliptic problems via the
  {M}orse theory and the domain topology}, Calc. Var. Partial Differential
  Equations \textbf{2} (1994), no.~1, 29--48.

\bibitem{BP05}
Jaeyoung Byeon and Junsang Park, \emph{Singularly perturbed nonlinear elliptic
  problems on manifolds}, Calc. Var. Partial Differential Equations \textbf{24}
  (2005), no.~4, 459--477.

\bibitem{Da88}
E.~N. Dancer, \emph{The effect of domain shape on the number of positive
  solutions of certain nonlinear equations}, J. Differential Equations
  \textbf{74} (1988), no.~1, 120--156.

\bibitem{DMP09}
Edward~Norman Dancer, Anna~Maria Micheletti, and Angela Pistoia,
  \emph{Multipeak solutions for some singularly perturbed nonlinear elliptic
  problems in a {Riemannian} manifold}, Manuscripta Math. \textbf{128} (2009),
  no.~2, 163--193.

\bibitem{deF89}
D.~G. de~Figueiredo, \emph{Lectures on the {E}keland variational principle with
  applications and detours}, Tata Institute of Fundamental Research Lectures on
  Mathematics and Physics, vol.~81, Published for the Tata Institute of
  Fundamental Research, Bombay, 1989.

\bibitem{H09}
Norimichi Hirano, \emph{Multiple existence of solutions for a nonlinear
  elliptic problem in a {Riemannian} manifold}, Nonlinear Anal. \textbf{70}
  (2009), no.~2, 671--692.

\bibitem{MP09c}
Anna~Maria Micheletti and Angela Pistoia, \emph{Generic properties of
  singularly perturbed nonlinear elliptic problems on {Riemannian} manifolds},
  Adv. Nonlinear Stud. \textbf{9} (2009), no.~4, 803--815.

\bibitem{MP09b}
Anna~Maria Micheletti and Angela Pistoia, \emph{Nodal solutions for a singularly perturbed nonlinear elliptic
  problem in a {Riemannian} manifold}, Adv. Nonlinear Stud. \textbf{9} (2009),
  no.~3, 565--577.

\bibitem{MP09}
Anna~Maria Micheletti and Angela Pistoia, \emph{The role of the scalar curvature in a nonlinear elliptic problem
  in a {Riemannian} manifold}, Calc. Var. Partial Differential Equations
  \textbf{34} (2009), no.~2, 233--265.

\bibitem{V08}
Daniela Visetti, \emph{Multiplicity of solutions of a zero-mass nonlinear
  equation in a {Riemannian} manifold}, J. Differential Equations \textbf{245}
  (208), no.~9, 2397--2439.

\end{thebibliography}

\end{document}